\documentclass[12pt]{article}

\usepackage{amsfonts}

\usepackage{amscd}
\usepackage{amsthm}
\usepackage{indentfirst}

\usepackage[all]{xy}

\usepackage{amssymb, amsmath, amsthm, amsgen, amstext, amsbsy, amsopn}
\usepackage{epsfig}
\usepackage{epic,eepic}

\newtheorem{thm}{Theorem}[section]
\newtheorem{lemma}[thm]{Lemma}
\newtheorem{prop}[thm]{Proposition}

\newtheorem{defin}[thm]{Definition}
\newtheorem{rem}[thm]{Remark}
\newtheorem{exam}[thm]{Example}

\newcommand{\R}{{\mathbb{R}}}

\newcommand{\N}{{\mathbb{N}}}
\newcommand{\C}{{\mathbb{C}}}

\newcommand{\cD}{{\mathcal{D}}}
\newcommand{\cC}{{\mathcal{C}}}

\newcommand{\cL}{{\mathcal{L}}}

\newcommand{\cN}{{\mathcal{N}}}
\newcommand{\cP}{{\mathcal{P}}}

\def\id{{1\hskip-2.5pt{\rm l}}}

\begin{document}

\title{Quantum unsharpness and symplectic rigidity}

\renewcommand{\thefootnote}{\alph{footnote}}

\author{\textsc Leonid
Polterovich }

\footnotetext[1]{ Partially supported by the National Science
Foundation grant DMS-1006610 and the Israel Science Foundation grant 509/07. }
\footnotetext[2]{{\it MSC classes:} 53D35,81S10,81P15  }

\date{\today}

\maketitle

\begin{abstract}
\noindent We discuss a link between ``hard" symplectic topology and
an unsharpness principle for generalized quantum observables (positive operator valued measures).
The link is provided by the Berezin-Toeplitz quantization.
\end{abstract}

\bigskip

\section{Introduction}\label{sec-intro} According to a version of the correspondence principle, quantum mechanics contains the classical one as a limiting case. Mathematically, this is modeled by quantization schemes which associate  Hermitian operators acting on Hilbert spaces (quantum observables) to functions on symplectic manifolds (classical observables), see e.g. \cite{Ali-survey} for a survey. Starting from the 1980ies,
``hard" methods in symplectic topology such as Morse theory on loop spaces and pseudo-holomorphic curves (see e.g. \cite{MS}) gave rise to discovery of various surprising rigidity phenomena taking place on symplectic manifolds. In the present note we propose a link between these developments and the Berezin-Toeplitz quantization (the latter is briefly reviewed in Section \ref{sec-BT} below).

On the symplectic side, our starting point is (a quantitative version of) the following phenomenon \cite{EPZ} which
is described in more details in Sections \ref{sec-sympprelim} and \ref{sec-rigpu}:
{\it A partition of unity associated to a sufficiently fine finite open cover
of a closed symplectic manifold cannot consist of Poisson-commuting functions.}
It turns out that the translation of this statement into the quantum language by means of the Berezin-Toeplitz
quantization fits well the discussion on unsharp observables appearing in the physical literature \cite{Busch}.
Here the observables are represented by positive operator valued measures (POVMs), and their unsharpness can be conveniently measured in terms of the noise operator (see Sections \ref{sec-povms}-\ref{sec-smearing}).
POVMs appearing in our context are the images of the partitions of unity under the Berezin-Toeplitz quantization,
and the above-mentioned non-commutativity phenomenon translates into the fact that the corresponding POVMs are subject to a {\it systematic intrinsic noise}. The precise formulation of this result and its mechanical interpretation are given in Sections \ref{sec-mainth} and \ref{sec-interpret} respectively. The paper is concluded with some open problems proposed in Section \ref{sec-disc}.

\section{Symplectic preliminaries} \label{sec-sympprelim} We start with some preliminaries
on symplectic geometry and topology (see \cite{MS-Intro, P-book} for more details). The phase space of classical mechanics is modeled
by {\it a symplectic manifold}, that is by an even-dimensional manifold $M^{2n}$ equipped
with a closed differential $2$-form $\omega$ whose top power $\omega^n$ does not vanish
at any point of $M$. In particular, $\omega^n/n!$ is the {\it canonical volume form}
on $M$. By the Darboux theorem, near each point
of a symplectic manifold one can choose local coordinates $p_1,q_1,...,p_n,q_n$ so that in these coordinates
$\omega = \sum_{j=1}^n dp_j \wedge dq_j$. Below we focus on closed (that is compact without boundary)  connected symplectic manifolds. The basic examples include closed oriented surfaces equipped with the area form, the complex projective space $\C P^n$ with the Fubini-Study symplectic form
(that is the unique symplectic form which is invariant under the action of the unitary group $U(n+1)$ and which integrates to $1$ over the projective line $\C P^1 \subset \C P^n$) and their products.

A mechanical system is described by its energy, that is a Hamiltonian function $f \colon M \times \R \to \R$. According to a basic principle of classical mechanics, the energy determines the time evolution of the system, in the following way: Consider the Hamiltonian system on $M$ which in the Darboux coordinates is given
by
\begin{equation*}
\begin{cases}
 \dot{q}_i = \displaystyle \hspace{9.5pt} \frac{\partial{f}}{\partial{p_i}}(p,q,t)  \vspace{5pt}\\
\dot{p}_i = \displaystyle - \frac{\partial{f}}{\partial{q_i}}(p,q,t)
\end{cases}
\end{equation*}
It gives rise to a one-parameter family $\phi_t$ of diffeomorphisms of $M$ which
send the initial condition $z(0)$ to the solution $z(t)$ at time $t$.
Diffeomorphisms $\phi_t$ are called {\it Hamiltonian diffeomorphisms}. Hamiltonian diffeomorphisms
preserve the symplectic form $\omega$ and hence the phase volume $\omega^n/n!$.

\medskip

A subset $U \subset M$ is called {\it displaceable} if there exists a Hamiltonian diffeomorphism
$\phi$ of $M$ so that $\phi(U) \cap \text{Closure}(U) = \emptyset$. We say that a subset $X \subset M$ is {\it dominated} by an open subset $U \subset M$ if $\phi(X) \subset U$ for some Hamiltonian diffeomorphism $\phi$ of $(M,\omega)$.  Clearly, every subset dominated by a displaceable one is itself displaceable.

The property of being displaceable is sensitive to symplectic geometry and topology of the set $U$. As an illustration, consider the sphere $S^2$ equipped with the standard area form of the total area $1$. Every open disc with smooth boundary of the area $\leq 1/2$ can be mapped to the upper hemisphere by an area-preserving map and hence is displaceable by a rotation. On the other hand, the equator (that is a simple closed curve dividing the sphere into two discs of the area $1/2$) is non-displaceable. A fortiori, any annulus of arbitrary small area containing an equator is non-displaceable.

Obstructions to displaceability on higher dimensional symplectic manifolds are much more delicate. They belong to the realm of symplectic rigidity phenomena mentioned in the introduction, and their study form one of the central directions of modern symplectic topology. For instance, consider the complex projective space $\C P^n$ equipped with the Fubini-Study symplectic form.  Let $[z_0:...:z_n]$ be the homogeneous coordinates on $\C P^n$. The {\it Clifford torus} $\{|z_0|=...=|z_n|\}$ in $\C P^n$ (which can be considered as a generalization
of the equator in $S^2=\C P^2$) is non-displaceable. This was proved in \cite{BirEP} by the methods
discussed in Section \ref{sec-rigpu} below and in \cite{Cho} via Lagrangian Floer theory.

\medskip

Next, let us turn to the space $C^{\infty}(M)$ of smooth functions on a closed symplectic manifold
$(M,\omega)$. It is equipped with {\it the Poisson bracket} $\{f,g\}$, which in the Darboux coordinates $(p,q)$ is given by $$\{f,g\} = \frac{\partial f}{\partial q} \cdot \frac{\partial g}{\partial p}-\frac{\partial f}{\partial p} \cdot \frac{\partial g}{\partial q}\;.$$ We say that $f$ and $g$ {\it Poisson commute}
if their Poisson bracket vanishes: $\{f,g\}=0$. Finite-dimensional linear subspaces of $C^{\infty}(M)$
consisting of Poisson commuting functions naturally arise in the theory of integrable Hamiltonian systems and
of Hamiltonian tori actions on symplectic manifolds.

 Write $||f||$ for the uniform norm $\max_{M} |f|$ of a function $f \in C^{\infty}(M)$. Consider the functional \begin{equation}\label{eq-Phi-def}
\Phi: C^{\infty}(M) \times C^{\infty}(M) \to \R,\;\; \Phi(f,g):=||\{f,g\}||\;.
 \end{equation}
The functional $\Phi$ is known to exhibit
various rigidity-type properties \cite{EP,EPZ,Buh,BEP}. I am focusing below on one of them known
as {\it rigidity of partitions of unity} \cite{EPZ} and in particular on its translation into the language
of quantum mechanics.

\section{Rigidity of partitions of unity} \label{sec-rigpu}
Let $M$ be a closed manifold. A {\it partition of unity} $\{f_j\}$, $j=1,...,N$
is a collection of non-negative smooth functions  $\{f_j\}$, $j=1,...,N$ on $M$ which sums up to $1$:
$$f_1 +...+f_N =1\;.$$ We say that a partition of unity  $\{f_j\}$ is subordinated to an
open cover $\{U_j\}$,  $j=1,...,N$ of $M$ if $\text{supp}(f_j) \subset U_j$. Here
$\text{supp}(f)$ stands for the support of the function $f$, that is for the closure of the set $\{f \neq 0\}$.

\medskip
Suppose now that $M$ is equipped with a symplectic form $\omega$.
Let $\{f_j\}$, $j=1,...,N$ be a partition of unity on $M$. It is called {\it Poisson commutative}
if the Poisson brackets $\{f_i,f_j\}$ vanish for all $i,j$, and {\it Poisson non-commutative} otherwise. Put
\begin{equation}\label{eq-nu-class}
\nu_c(\{f_j\}) = \max \Phi\Big{(}\sum x_j f_j,\sum y_k f_k\Big{)}\;,
\end{equation}
where the maximum is taken over all vectors $x=(x_1,...,x_N)$ and $y = (y_1,...,y_N)$ from
the cube $K_N := [-1,1]^N$, and $\Phi$ is given by \eqref{eq-Phi-def}. The quantity $\nu_c$ measures classical (hence the subindex $c$) Poisson non-commutativity of the partition $\{f_j\}$.

\medskip
\noindent
\begin{thm} \label{thm-partunity} (cf. \cite{EPZ}). For every displaceable open subset
$U \subset M$ there exists a positive constant
$C$ so that
\begin{equation}
\label{eq-partunity}
\nu_c(\{f_j\}) \geq \frac{C}{N^2}
\end{equation}
for any partition of unity $\{f_j\}$, $j=1,...,N$ subordinated
to a cover of $M$ by $U$-dominated open subsets.
\end{thm}

\medskip
\noindent For certain symplectic manifolds
(for instance, for complex projective spaces) it suffices to require that the cover consists of displaceable subsets (not necessarily dominated by the same displaceable subset $U$), and the constant $C$ depends only
on $(M,\omega)$, see \cite{EPZ}. In this form, rigidity of partitions of unity can be used for proving non-displaceability of certain subsets which cannot be detected by the methods of classical differential geometry and topology. For instance, it yields non-displaceability of the Clifford torus in $\C P^n$ which has been already mentioned in Section \ref{sec-sympprelim} above (we refer to \cite{BirEP} for an elementary geometric argument).

It is currently unknown whether the asymptotic behavior $\sim N^{-2}$ in inequality
\eqref{eq-partunity} is optimal: this is a direction of an ongoing research.
Let us mention also that according to \cite[Theorem 3.6]{RS}, every closed $2n$-dimensional symplectic manifold can be covered by $2n+1$ displaceable subsets.

\medskip
\noindent
\begin{exam}\label{exam-sphere} {\rm Consider the unit sphere $S^2 = \{q_1^2 +q_2^2 +q_3^2 =1\} \subset \R^3$ equipped with the
standard area form. Take any open cover $\{V_1,...,V_N\}$ of the interval $[-1,1]$,
and put $$U_j = \{(q_1,q_2,q_3) \in S^2 \;:\; q_3 \in V_j\}\;.$$
The cover $\{U_j\}$ of the sphere admits a subordinated partition of unity of the form $\{f_j(q_3)\}$,
which is Poisson commutative. Of course, there is no contradiction with Theorem \ref{thm-partunity}:
Indeed, at least one of the sets $U_j$ necessarily contains
the equator $\{q_3=0\}$ and hence is not displaceable.}
\end{exam}

\medskip
\noindent The proof of Theorem \ref{thm-partunity}  is analogous to the one
of section 2.2 of \cite{EPZ} which deals with a special class of symplectic manifolds.
A recent work by Usher \cite{U} yields that this argument is in fact
applicable to all closed symplectic manifolds. The key tool used in the proof is the functional $\zeta: C^{\infty}(M) \to \R$ provided by Hamiltonian Floer theory:
$$\zeta(f) = \lim_{s\to\infty} \frac{c([M],sf)}{s}\;,$$
where $c([M],.)$ is the Floer-homological spectral invariant associated with the fundamental
class $[M]$ of $M$ (see \cite{Schwarz,Oh}). Floer theory is the Morse-Novikov theory
of the classical action functional $\int fdt-pdq$ on the loop space of the symplectic manifold $M$,
and the spectral invariant $c([M],f)$ is a critical value of the action functional obtained via a suitable homological minmax (see \cite[Chapter 12]{MS} for a detailed exposition).

\medskip
\noindent
The following properties of $\zeta$ are relevant for us \cite{EP-qs}:
\begin{itemize}
\item[{(i)}] $\zeta(1)=1$;
\item[{(ii)}] $\zeta(f)=0$ provided the support of $f$ is displaceable;
\item[{(iii)}] $\zeta(f) \geq 0$ if $f \geq 0$.
\end{itemize}
Furthermore, as it is explained in \cite{EPZ} we have that
\begin{equation}\label{eq-master}
|\zeta(f+g) -\zeta(f) -\zeta(g)| \leq \sqrt {C^{-1}\cdot ||\{f,g\}||}
\end{equation}
for all functions $f,g$ provided the support of $g$ is dominated by an open displaceable subset $U$.
Here the constant $C$ depends only on $U$.

\medskip
\noindent
 The functional $\zeta$ is an example of
{\it a partial symplectic quasi-state}, a notion introduced in \cite{EP-qs}. Partial symplectic quasi-states
are non-linear functionals on $C^{\infty}(M)$ which, roughly speaking, are characterized by properties
(i)-(iii) and \eqref{eq-master} above. They  play a basic role in function theory on symplectic manifolds \cite{EPZ,BEP} and serve as a useful tool for a number of problems in symplectic topology. The origins of this notion go back to foundations of quantum mechanics \cite{EP-qs,EPZ-physics} (see also a brief discussion in Section \ref{sec-disc} below). We refer to \cite{EP-qs} for an axiomatic definition and more examples of partial symplectic quasi-states.

\medskip
\noindent
{\bf Proof of Theorem \ref{thm-partunity}:}   Put $t=\nu_c(\{f_j\})$ and $g_k = \sum_{j=1}^k f_k$.
Note that by (ii) $\zeta(g_1)=0$. Since $g_{k+1}=g_{k} + f_{k+1}$,
we have that by (ii),(iii) and \eqref{eq-master}
$$\zeta(g_{k+1}) \leq \zeta(g_k) + \sqrt{C^{-1}t}\;.$$
Thus by (i),
$$1=\zeta(g_N) \leq N\sqrt{C^{-1}t}\;,$$
which yields $t \geq C/N^2$, as required.
\qed

\section{Preliminaries on POVMs}\label{sec-povms} Let $H$ be a complex Hilbert space. Denote by $\cL(H)$ the space of all bounded Hermitian operators on $H$. Consider a set $\Omega$ equipped with a $\sigma$-algebra $\cC$ of its subsets. An $\cL(H)$-valued {\it positive operator valued measure} (POVM) $F$ on $(\Omega,\cC)$ is a countably additive map $F: \cC \to \cL(H)$ which takes a subset $X \in \cC$ to a positive operator $F(X) \in \cL(H)$ and which is normalized by $F(\Omega) = \id$.

An important class of POVMs is formed by {\it projection valued measures} for which all the operators
$F(X)$, $X \in \cC$ are orthogonal projectors. In this case $F(X \cap Y)= F(X)F(Y)$ for
any pair of subsets $X,Y \in \cC$.

POVMs naturally appear in quantum measurement theory \cite{Busch} where they play a role
of generalized observables. The space $\Omega$ is called
the {\it value space} of the observable. Pure states of the system are represented by the points
of the projective space $[\xi] \in \mathbb{P} (H)$, where $\xi \in H$ is a unit vector. When the system is in a state $[\xi]$, the probability of finding the observable $F$ in a subset $X \in \cC$ is postulated to be
$\langle F(X)\xi,\xi \rangle$. With this language, projection valued measures correspond to {\it sharp}
observables.

It is instructive to compare this model with the traditional von Neumann quantum mechanics where an
observable is given by a bounded Hermitian operator, say $A$ on $H$. For our purposes it suffices to restrict ourselves to the case when $H$ is finite dimensional. Look at the spectral decomposition
$$A = \sum_{j=1}^N \lambda_j P_j$$
of $A$. Here $\{\lambda_j\}$ are pair-wise distinct eigenvalues of $A$ and $P_j$
is the orthogonal projector to the eigenspace corresponding to $\lambda_j$. If the system is in a state
$[\xi]$, the observable $A$ takes values $\lambda_j$ with probability $\langle P_j \xi,\xi\rangle$.
Thus the observable $A$ is fully described by the projector-valued measure  $\sum P_j\delta_{\lambda_j}$ on $\R$ where $\delta$ stands for the Dirac delta function.

We shall often deal with POVMs on the finite set $\Omega_N=\{1,...,N\}$.
Such a POVM, say $A$, is fully determined by a collection of operators $A_j:= A(\{j\})$, $j=1,...,N$.
Vice versa, any collection $\{A_j\}$, $j=1,...,N$ of positive bounded Hermitian operators with
$\sum_j A_j =\id$ defines a POVM on $\Omega_N$.

\section{Unsharpness and noise}\label{sec-unsharpness}
Here we describe an approach to unsharpness  of POVMs based on the notion of the {\it (intrinsic) noise operator}, see \cite[\S 2]{BHL1}, \cite[\S 4]{Ozawa}, \cite[\S 3]{BHL2}, \cite[\S 2]{Massar}.
(Some authors refer to intrinsic uncertainty instead of intrinsic noise.) The definition is especially
transparent in the case of an $\cL(H)$-valued POVM $A=\{A_1,...,A_N\}$ on the finite set $\Omega_N=\{1,...,N\}$, where $H$ is a finite-dimensional Hilbert space. Given a vector $x=(x_1,...,x_N) \in \R^N$, consider a POVM
$\widehat{A}(x) = \sum A_j \cdot \delta_{x_j}$ on $\R$ as well as a Hermitian operator
\begin{equation}\label{eq-Ax}
A(x) = \sum x_j A_j
\end{equation}
on $H$.
Assume that the system is prepared in a state $[\xi]$, where $\xi$ is a unit vector in $H$.
The POVM-observable $\widehat{A}(x)$ is described by the random variable $\phi$ which takes value
$x_j$ with probability $\langle A_j\xi,\xi\rangle$. The von Neumann observable $A(x)$ is described by a random variable $\psi$
which takes value $\lambda_j$ with probability $\langle P_j\xi,\xi \rangle$, where
$A(x) = \sum \lambda_j P_j$ is the spectral decomposition of $A(x)$.
Comparing the expectations, we see that
$$\mathbb{E}\psi = \sum \lambda_j \cdot \langle P_j\xi,\xi \rangle = \langle A(x) \xi,\xi\rangle = \sum x_j \langle A_j \xi,\xi\rangle = \mathbb{E}\phi\;.$$
Further,
$$\mathbb{E}\psi^2= \sum \lambda_j^2 \cdot \langle P_j\xi,\xi \rangle =  \langle A(x)^2 \xi,\xi\rangle\;,$$
and
$$\mathbb{E}\phi^2 = \Big{\langle} \Big{(}\sum x_j^2 A_j\Big{)} \xi,\xi \Big{\rangle}\;.$$
The noise operator is defined as
\begin{equation}\label{eq-delta-def}
\Delta_A (x):= \sum_{j=1}^N x_j^2 A_j- A(x)^2 = \sum_{j=1}^N (A(x)-x_j\id)A_j(A(x)-x_j\id)\;,
\end{equation}
so that the difference of the variances of $\phi$ and $\psi$ is given by
$$\text{Var}(\phi) -\text{Var}(\psi) = \langle \Delta_A (x)\xi,\xi\rangle\;.$$
The noise operator has a number of interesting properties:
In particular, $\Delta_A(x)$ is always non-negative: look at the second equality in \eqref{eq-delta-def} which we learned from \cite{Massar}. In addition, $\Delta_A(x)\equiv 0$ if and only if
$A$ is a projection valued measure, that is a sharp observable. In fact, for an unsharp $A$,  $\Delta_A(x) \neq 0$ provided all $x_j$'s are pair-wise distinct, see e.g. \cite[Proposition 7]{KLY}.

The above discussion leads us to the following definition: given a POVM $A$ on $\Omega_N$,
we measure its unsharpness through the {\it magnitude of noise}
$$\cN(A):= \max_{x \in K_N} ||\Delta_A(x)||_{op}\;,$$
where $K_N$ is the cube $[-1,1]^N$. Observe that
$$0 \leq \Delta_A(x) \leq \sum x_j^2 A_j \leq \sum A_j=\id$$
for all $x \in K_N$, and therefore
\begin{equation}\label{eq-interval-u}
0 \leq \cN(A) \leq 1
\end{equation}
for all POVMs $A$.

It turns out that the magnitude of noise of a POVM $A$
can be estimated through the degree of non-commutativity of $A$ which is defined as follows
(compare with formula \eqref{eq-nu-class} above):
For $x =(x_1,...,x_N) \in \R^N$ consider the operator
$A(x):= \sum_j x_j A_j$ which already appeared in \eqref{eq-Ax} above.
Put
\begin{equation}\label{eq-nuq}
\nu_q(A) := \max_{x,y \in K_N} ||[A(x),A(y)]||_{op}
\end{equation}
(here subindex $q$ stands for quantum).
The next result is a minor modification of a theorem by Janssens \cite[Theorem 3]{Janssens} (cf. \cite[Corollary 2]{MI}):

\medskip
\noindent \begin{thm}\label{thm-uncert}
\begin{equation}\label{eq-ineq-uncert}
\cN(A) \geq \frac{1}{2}\nu_q(A)\;.
\end{equation}
\end{thm}

\medskip
\noindent
Theorem \ref{thm-uncert} is an immediate consequence of the following lemma:

\medskip
\noindent\begin{lemma} (cf. \cite{Janssens}, see also \cite{MI}).
$$||\Delta_A(x)||^{1/2}_{op}\cdot ||\Delta_A(y)||^{1/2}_{op} \geq \frac{1}{2} \cdot ||\;[A(x),A(y)]\;||_{op}\;$$
for all $x,y \in K_N$.
\end{lemma}

\begin{proof} Let $A=\{A_1,...,A_N\}$ be an $\cL(H)$-valued POVM on $\Omega_N$.
By Naimark's theorem (see e.g. \cite{Busch}) there exists a Hilbert space $H'$ containing $H$,
and a projector valued measure $\{P_1,...,P_N\}$ on $H'$ so that $A_i = \Psi(P_i)$,
where
$$\Psi(B) := \Pi B \Pi^* \;\; \forall B \in \cL(H')\;,$$
and $\Pi: H' \to H$ is the orthogonal projector. Janssens  \cite[Theorem 3]{Janssens} showed that
for every pair $B_1,B_2$ of commuting Hermitian operators on $H'$
\begin{equation}\label{eq-Janss}
||\Psi(B_1^2) -\Psi(B_1)^2||_{op}^{1/2} \cdot ||\Psi(B_2^2) -\Psi(B_2)^2||_{op}^{1/2} \geq
\frac{1}{2}\cdot ||\;[\Psi(B_1),\Psi(B_2)]\;||_{op}\;.
\end{equation}
Applying now \eqref{eq-Janss} with $B_1 =\sum x_i P_i$ and $B_2=\sum y_i P_i$
and observing that
$$\Psi(B_1) = A(x),\;\; \Psi(B_1^2)-\Psi(B_1)^2 = \Delta_A(x)\;,$$
$$\Psi(B_2) = A(y),\;\; \Psi(B_1^2)-\Psi(B_1)^2 = \Delta_A(y)\;,$$
we get the desired inequality.
\end{proof}

\medskip
\noindent \begin{rem}\label{rem-general} {\rm The results of the present section and their proofs readily extend to $\cL(H)$-valued
POVMs on $(\Omega,\cC)$ where $H$ is an infinite-dimensional Hilbert space and $(\Omega,\cC)$ is
an arbitrary set equipped with a $\sigma$-algebra of subsets. Since we shall face such a degree of generality
in Section \ref{sec-interpret} below, let us briefly discuss the corresponding definitions and statements.
Let $A: \cC \to \cL(H)$ be a POVM.
Denote by $K(\Omega)$ the set of measurable functions $x:\Omega \to \R$ with $\max |x| \leq 1$.
(Observe that for the finite set $\Omega_N$ every function $x$ from $K(\Omega_N)$ is canonically identified with
a vector $x=(x(1),...,x(N))$ lying in the cube $K_N = [-1,1]^N$.)
For $x \in K(\Omega)$ put $A(x):= \int x\; dA$ and define the noise operator $\Delta_A(x) = \int_\Omega x^2 dA - A(x)^2$. Define the magnitude of noise
$\cN(A) = \sup_{x \in K(\Omega)} ||\Delta_A(x)||_{op}$ and
the magnitude of non-commutativity
$\nu_q(A) = \sup_{x,y \in K(\Omega)} ||[A(x),A(y)]||_{op}$. Inequality \eqref{eq-ineq-uncert}
remains valid for this more general setting. }\end{rem}

\medskip

Let us mention that every sharp observable is commutative, but not vice versa.
Consider, for instance, a POVM $A_1 = ... = A_N = (1/N)\cdot \id $,
where $N$ is even. Take a vector $x_* \in K_N$ such that half of its coordinates are equal
to $1$ and half to $-1$. Then $A(x_*)=0$ and $\Delta_A(x_*) = \id$. Together with \eqref{eq-interval-u}
this yields $\cN(A)=1$, while $\nu_q(A) =0$ since $A$ is commutative. We see that inequality  \eqref{eq-ineq-uncert} is far from being optimal: the commutative POVM $A$ has the maximal possible magnitude of noise. In the next section we refine inequality \eqref{eq-ineq-uncert} in order to fix this problem.

\section{The effect of smearing}\label{sec-smearing}
In this section we deal with $\cL(H)$-valued POVMs defined on Hausdorff locally compact second countable topological spaces. (For the moment, the Hilbert space $H$ is not assumed to be finite-dimensional.)
Let $\Omega$ and $\Theta$ be such topological spaces and let $\cC,\cD$ be their Borel $\sigma$-algebras respectively. Denote by $\cP(\Omega)$ the set of Borel probability
measures on $\Omega$. A {\it Markov kernel} is a map
$$\gamma: \Theta \to \cP(\Omega),\; w \mapsto \gamma_w$$
such that the function $w \to \gamma_w(X)$ on $\Theta$ is measurable for every $X \in \cC$.
Let $A$ and $B$ be POVMs on $(\Omega,\cC)$ and $(\Theta, \cD)$ respectively. We say that
$A$ is a {\it smearing}\footnote{Some authors call it randomization or fuzzification.} of $B$
\cite{Busch, JP, Ali}
if there exists a Markov kernel $\gamma$ so that
$$A(X) = \int_\Theta \gamma_w(X)\; dB(w)\;\; \forall X \in \cC\;.$$
In the physical language, each element $w$ of the value set $\Theta$ of $B$ diffuses into
a subset $X \in \cC$ with probability $\gamma_w(X)$.

\medskip

For instance, if $\Omega= \Omega_N=\{1,...,N\}$,
a Markov kernel $\gamma$ is given by a collection of non-negative measurable functions
$\gamma_j$ on $\Theta$, $j=1,...,N$ so that $\sum_j \gamma_j(w)=1$ for all $w \in \Theta$,
that is by a measurable partition of unity.
A POVM $A=\{A_1,...,A_N\}$ on $\Omega_N$ is a smearing
of a POVM $B$ on $(\Theta,\cD)$ if
\begin{equation} \label{eq-smearing-int}
A_j = \int_\Theta \gamma_j dB\;.
\end{equation}
A specific example
of such a situation will arise in Section \ref{sec-interpret} below.

\medskip

It turns out that the magnitude of non-commutativity $\nu_q$ defined in \eqref{eq-nuq} above
behaves monotonically with respect to smearing:

\medskip
\noindent
\begin{prop}\label{prop-smearings} Assume that $A$ is a smearing of $B$.
Then $\nu_q(B) \geq \nu_q(A)$.
\end{prop}

\begin{proof} We work in the notation of Remark \ref{rem-general}.
Define an affine map $\Gamma: K(\Omega) \to K(\Theta)$ by
$$(\Gamma x)(w) = \int_\Omega x\;d\gamma_w\;.$$ One readily verifies that
$$A(x) = \int_\Omega x\; dA = \int_\Theta (\Gamma x) \; dB = B(\Gamma x)\;.$$
Thus
$$\nu_q(B) \geq \sup_{x,y \in K(\Omega)} ||[B(\Gamma x),B(\Gamma y)]||_{op}
= \sup_{x,y \in K(\Omega)} ||[A(x),A(y)]||_{op}=\nu_q(A)\;.$$
\end{proof}

\medskip

Let $A$ and $B$ be two POVMs so that $A$ is a smearing of $B$. Recall that smearing can be interpreted
as a diffusion from the value set of $B$ to the one of $A$. Therefore, if the noise $\cN(B)$ is strictly
less than $\cN(A)$, we think of the increment $\cN(A)-\cN(B)$
as of a random component of the noise of $A$. An attempt to extract the systematic (as opposed to the random)
component of the intrinsic noise of a POVM $A$ leads to the following definition:

\medskip
\noindent
\begin{defin}\label{def-pers}{\rm A {\it systematic noise} of a POVM $A$ is defined as $$\cN_s(A) = \inf \cN(B)\;,$$ where the infimum is taken over all POVMs $B$ so that $A$ is a smearing of $B$ (subindex $s$ stands for systematic).
}\end{defin}

\medskip

 It is known \cite[Section 5]{Ali} (cf. \cite{JP}) that every commutative POVM on a Hausdorff locally compact second countable space is necessarily a smearing of a sharp observable, that is of a projector valued measure. In particular, $\cN_s(A)=0$ provided $\nu_q(A)=0$. Our next result shows that the converse statement is also true:

\medskip
\noindent
\begin{thm}\label{cor-uncert-sm} (Unsharpness principle for POVMs).
$$\cN_s(A) \geq  \frac{1}{2}\nu_q(A)\;$$
for every $\cL(H)$-valued POVM $A$ on $\Omega_N$.
\end{thm}

\medskip
\noindent Indeed, combining Theorem \ref{thm-uncert} with Proposition \ref{prop-smearings} we get that if $A$ is a smearing of $B$,
$$\cN(B) \geq \frac{1}{2}\nu_q(B) \geq \frac{1}{2}\nu_q(A)\;,$$
which yields the desired inequality.

\medskip
\noindent Let us mention also that the magnitude of noise $\cN(A)$ does not behave in a definite way under smearing: in general, it may either increase or decrease. Theorem \ref{cor-uncert-sm} provides a constraint for a possible decrease in case when POVM $A$ is non-commutative.

\section{Berezin-Toeplitz quantization}\label{sec-BT} Recall the general scheme of the Berezin-Toeplitz quantization. As above, we denote by $\cL(H)$ the space of bounded Hermitian operators acting on a complex Hilbert space $H$.
Let $(M,\omega)$ be a closed symplectic manifold. The Berezin-Toeplitz quantization
consists of a sequence of finite-dimensional complex Hilbert spaces $H_m$, $m \to \infty$ of the increasing dimension and a family
of surjective $\R$-linear maps $T_m: C^{\infty}(M) \to \cL(H_m)$ with the following properties:
\begin{itemize}
\item[{(BT1)}] $T_m(1) =\id$;
\item[{(BT2)}] $T_m(f) \geq 0$ for $f \geq 0$;
\item[{(BT3)}]  $||T_m(f)||_{op} = ||f||+ O(1/m)$;
\item[{(BT4)}] $||i \cdot m[T_m(f),T_m(g)] - T_m(\{f,g\})||_{op} = O(1/m)$;
\item[{(BT5)}]  $||T_m(f^2)-T_m(f)^2||_{op} = O(1/m)$
\end{itemize}
as $m \to \infty$ for all $f,g \in C^{\infty}(M)$.
Here $||f||$ stands for the uniform norm of a function $f \in C^{\infty}(M)$,
$\{f,g\}$ for the Poisson bracket, $||A||_{op}$ for the operator norm of $A \in \cL(H)$ and
$[A,B]$ for the commutator $AB-BA$. The number $m$ plays the role of the quantum number,
while the Planck constant $\hbar$ equals $2\pi/m$,
so that $m \to \infty$ corresponds to the classical limit (see \cite{classiclimit} for an illuminating
discussion on the classical limit). Let us emphasize that the remainders $O(1/m)$ in the formulas above  do depend on the norms of derivatives of functions $f$ and $g$.

Existence of the Berezin-Toeplitz quantization is a highly non-trivial fact which goes back to the pioneering
work by Berezin \cite{Berezin} who succeeded to quantize certain symplectic homogeneous spaces such as
the complex projective space. At the moment existence of the Berezin-Toeplitz quantization is known
for all closed symplectic manifolds whose symplectic form $\omega$ represents an integral cohomology class.
This was established in \cite{BMS} (see also \cite{S} for a survey) for closed K\"{a}hler manifolds
by using pseudo-differential calculus of Toeplitz operators developed in the classical monograph \cite{BdMG}
by Boutet de Monvel and Guillemin, and in \cite{Gu,BU,MM} for general symplectic manifolds.

In the case of closed K\"{a}hler manifolds with the integral symplectic form, the Berezin-Toeplitz quantization admits the following simple description:
Given an integral K\"{a}hler manifold $(M^{2n},\omega,J)$,
choose a holomorphic Hermitian line bundle $(L,h)$ over $M$ so that the curvature of its (unique)
Hermitian connection compatible with the holomorphic structure equals $-i\omega$ (existence of $(L,h)$
is a well known fact from complex algebraic geometry, see e.g. \cite[Section 7.1.3]{Voisin}).
Let $V_m$ be the $L^2$-space of sections of $L^{\otimes m}$ equipped with the scalar product
$$\langle s_1,s_2 \rangle:= \int_M h^{\otimes m}(s_1,s_2) \omega^n\;.$$ Define $H_m \subset V_m$
as the space of all holomorphic sections of $L^{\otimes m}$. Write $\Pi$ for the orthogonal projection $V_m \to H_m$. For a function $f \in C^{\infty}(M)$ denote by $S_f:V_m \to V_m$ the multiplication operator
$s \mapsto fs$. With this notation the operator $T_m(f)$ is defined as $\Pi S_f \Pi^* : H_m \to H_m$.
Properties (BT1) and (BT2) are obvious, while (BT3),(BT4) and (BT5) require a delicate analysis.

Property (BT4) reflects {\it the quantum-classical correspondence principle}:
the Poisson bracket of a pair of classical observables corresponds to $2\pi i/\hbar$ times the
commutator of their quantum counterparts up to terms of the order $O(\hbar)$.

It readily follows from properties (BT3) and (BT4) of the Berezin-Toeplitz quantization
that the functional
$$\Phi: C^{\infty}(M) \times C^{\infty}(M) \to \R,\;\; \Phi(f,g):=||\{f,g\}||$$
defined in \eqref{eq-Phi-def} above can be seen in the quantum limit:
\begin{equation}\label{eq-pbquantbound}
||m[T_m(f),T_m(g)]\;||_{op} = ||T_m(\{f,g\})||_{op}+ O(1/m) = ||\{f,g\}|| + O(1/m)
\end{equation}
for all smooth functions $f,g$ on $M$.

     The Berezin-Toeplitz quantization can be described in the language of POVMs : There exists a sequence of $\cL(H_m)$-valued POVMs $G_m$ on the symplectic manifold $M$ equipped with the Borel $\sigma$-algebra so that
\begin{equation}\label{eq-POVM_Toeplitz}
T_m(f)=\int_M f\; dG_m\;.
\end{equation}
For the sake of completeness, let us sketch the argument following the proof of Proposition 1.4.8 of Chapter II in \cite{L} (see also a discussion in \cite{MT}). Indeed, for every vector $\xi \in H_m$ define a linear real-valued functional on $C^{\infty}(M)$ by
$f \mapsto \langle T_m(f)\xi,\xi \rangle$.
Since every positive linear functional on $C^{\infty}(M)$ is given by the integration against
a Borel measure, there exists a measure $\mu^{(m)}_\xi$ on $M$ with
$ \langle T_m(f)\xi,\xi \rangle = \int f \;d\mu^{(m)}_\xi$.
The POVM $G_m$ is defined by the equality $\langle G_m(X)\xi,\xi \rangle = \mu^{(m)}_\xi (X)$
for every Borel subset $X \subset M$ and vector $\xi \in H_m$. In some sense we shall study below the deviation of POVMs $G_m$ from projector valued measures. Let us mention also that in the case of K\"{a}hler manifolds it is not hard to express POVMs $G_m$ in terms of coherent states (see \cite{Berezin,L,S}).

The following elementary observation is crucial for our purposes: Let $f_1,...,f_N$ be a partition of unity on a closed symplectic manifold $(M,\omega)$. Then the Berezin-Toeplitz quantization takes it into an $\cL(H_m)$-valued POVM $\{T_m(f_j)\}$ on $\Omega_N$. Indeed, the operators $T_m(f_j)$ are Hermitian, positive
and their sum equals $\id$.

\section{The main theorem}\label{sec-mainth} Finally, we are ready to present our main result. Consider a closed symplectic manifold $(M,\omega)$ whose symplectic form represents an integral de Rham
cohomology class. Fix a scheme of the Berezin-Toeplitz quantization
$T_m: C^\infty(M) \to \cL(H_m)$. Let $U \subset M$ be an open displaceable subset
of a closed symplectic manifold $(M,\omega)$. Take any partition of unity  $\{f_j\}$, $j=1,...,N$
subordinated to a cover of $M$ by $N$ open subsets dominated by $U$. Consider the  $\cL(H_m)$-valued POVM
$A^{(m)} = \{T_m(f_j)\}$ on $\Omega_N$ obtained from  $\{f_j\}$ by the Berezin-Toeplitz quantization.
Take any positive $c < C/2$, where $C$ is the constant provided by Theorem \ref{thm-partunity}.

\begin{thm}\label{thm-main} The systematic noise of POVM $A^{(m)}$ satisfies
\begin{equation}
\label{eq-POVM-1}
\cN_s(A^{(m)}) \geq \frac{c}{mN^2}
\end{equation}
for all sufficiently large $m \in \N$.
\end{thm}

\medskip
\noindent Let us emphasize that the constant $c$ depends only on the symplectic manifold $(M,\omega)$ and the domain $U$, but not on the specific cover $\{U_j\}$ and the partition of unity $\{f_j\}$.
Inequality \eqref{eq-POVM-1} holds for all $m \geq m_0$ where the number $m_0$ does depend on $\{f_j\}$.

\medskip
\noindent
Inequality \eqref{eq-POVM-1} can be spelled out as follows:  if $B^{(m)}$ is any sequence of POVMs whose smearing gives rise to $A^{(m)}$, POVM $B^{(m)}$ is necessarily unsharp with $\cN(B^{(m)}) \geq c/(mN^2)$ for all sufficiently large $m$. In contrast to $A^{(m)}$, the sequence of POVMs $B^{(m)}$  in general does not admit any definite classical counterpart. Therefore the link between rigidity of partitions of unity and quantum unsharpness presented above goes slightly beyond the direct translation via the correspondence principle.

\medskip
\noindent
{\bf Proof of Theorem \ref{thm-main}:} Theorem \ref{thm-partunity} and formula \eqref{eq-pbquantbound}
imply that
\begin{equation}
\label{eq-POVM}
\nu_q(A^{(m)}) \geq \frac{2c}{mN^2}
\end{equation}
for all sufficiently large $m$. Thus the unsharpness principle provided by Theorem \ref{cor-uncert-sm}
yields \eqref{eq-POVM-1}.
\qed

\medskip
\noindent
\begin{exam} \label{exam-commutBT} {\rm It follows from the discussion after Theorem \ref{thm-partunity}
above that for certain symplectic manifolds (such as complex projective spaces) Theorem \ref{thm-main}
remains valid when the cover consists of displaceable subsets (not necessarily dominated by the same displaceable subset $U$), in which case the constant $c$ depends only on $(M,\omega)$.
The assumption that the elements of the cover are displaceable cannot be lifted. Indeed, consider the cover of the sphere $S^2$ presented in Example \ref{exam-sphere} above
together with the subordinated partition of unity $\{f_j(q_3)\}$. It turns out that the corresponding Toeplitz operators $T_m(f_j)$ pair-wise commute. To check this claim, one can use the explicit model of the Berezin-Toeplitz quantization used by Berezin
in \cite{Berezin}. In this model the Hilbert space $H_m$ consists of complex polynomials
$p(z)$ of degree $\leq m$ and the function $q_3$ corresponds to the function $$u(z) = \frac{1-|z|^2}{1+|z|^2}\;, z \in \C= S^2\setminus \{\infty\}\;.$$ A direct calculation with Berezin's coherent states shows that the level $m$ Toeplitz operator
corresponding to any function of the form $f(u(z))$ has the eigenbasis $1,z,...,z^m$, which yields the claim.
Since the POVM $A^{(m)}=\{T_m(f_j)\}$ is commutative, the systematic noise $\cN_s(A^{(m)})$ vanishes for all $m$. There is no contradiction with Theorem \ref{thm-main} since one of the subsets of our cover
contains the equator of the sphere and hence is not displaceable.

Further, assume that this partition of unity
is non-trivial, that is one of the functions, say, $f_1$ attains a value $a$ which lies strictly between $0$ and $1$. It turns out that the magnitude of noise $\cN(A^{(m)})$ of $A^{(m)}$ (as opposed to the systematic
noise) remains bounded away from $0$ in the classical limit $m \to \infty$. In fact
we claim that given any positive constant  $\alpha < a-a^2$,
$$\cN(A^{(m)}) \geq \alpha$$ for all sufficiently large $m$. Indeed, take a vector $x_* = (1,0,...,0)$ in the cube $K_N=[-1,1]^N$. Then
$$\Delta_{A^{(m)}}(x_*) = T_m(f_1)-T_m(f_1)^2 \;.$$
It follows from property (BT5) of the Berezin-Toeplitz quantization that
for all sufficiently large $m$
$$||T_m(f_1)-T_m(f_1)^2||_{op} = ||f_1 -f_1^2|| + O(1/m) \geq \alpha\;, $$
and hence $\cN(A^{(m)}) \geq \alpha$ which proves the claim.
}\end{exam}

\section{Interpretation: what has been quantized?}\label{sec-interpret} Let $\{U_1,...,U_N\}$ be a cover of a set $M$. Consider the following process: Each point $z \in M$ has to register in exactly one set $U_j$  containing it. In other words, one has
to assign to each $z \in M$ an index $j \in \Omega_N :=\{1,...,N\}$ so that $z \in U_j$.
This task is in general ambiguous because of the overlaps between the subsets of the cover.
In order to resolve the ambiguity, the assignment is made at random: one chooses an index $j \in \Omega_N$ corresponding to a point $z \in M$ with probability $f_j(z)$, where
\begin{equation}\label{eq-prob}
f_j(z) = 0 \;\;\; \text{provided}\;\;\; z \notin U_j\;.
 \end{equation}
The latter condition guarantees a correct answer to the question {\it `Where (i.e. in which set $U_j$) is a given point $z$ located?'} Each specific outcome of the registration provides ``truth, but not the whole truth".

As an illustration \footnote{Thanks to Iosif Polterovich for this idea.}, consider the following toy model of a cellular communication network consisting of a collection of access points $u_1,...,u_N$. Each access point $u_j$ can be reached from a domain $U_j$, the so called {\it location area}. The location areas $U_j$ cover some territory $M$ considered, for simplicity, as a plane domain. Each cell phone at a given location $z \in M$  must register in exactly one access point $u_j$ whose location area $U_j$ contains $z$. When $z$ lies on the overlap between several
location areas, the choice between them is made at random with a probability depending on $z$ (in reality, of course, the choice involves more delicate considerations, see \cite{overlap}).

Suppose now that $M$ is a topological space, the subsets $U_j$ are open and the probabilities $f_j(z)$ are continuous functions. Let us slightly strengthen condition \eqref{eq-prob} above by assuming that
that $f_j$ vanishes in {\it a neighborhood} of the closed set $M \setminus U_j$. This means that $f_j$ is
supported in $U_j$. Since the probabilities $f_j(z)$, $j=1,...,N$ are non-negative and sum up to $1$ for every $z$, the functions $f_j$ form a partition of unity subordinated to the cover $\{U_j\}$. Given a  probability measure $\sigma$ on $M$, a randomly chosen (with respect to $\sigma$) point of $M$ is registered in the set $U_j$ with
probability
\begin{equation}\label{eq-prob-1}
\int_M f_j(z) \; d\sigma \;.
\end{equation}

Next, let us discuss this registration procedure in the context of {\it the Hilbert space model of classical mechanics} \cite[p.1628]{Misra}. Let $M$ be a compact classical phase space whose phase volume is given by a probability measure $\mu$. Consider the Hilbert space $H=L_2(M,\mu)$ of square-integrable complex valued functions on $M$. The states of the system are represented by unit vectors $\xi \in H$. Each such vector gives rise to
 a probability measure $d\sigma_\xi = |\xi|^2 \cdot d\mu$ on $M$. Take any partition of unity
 $\{f_1,...,f_N\}$  subordinated to an open cover  $U_1,...,U_N$ of $M$ and consider the registration
 procedure associated to this data. According to formula \eqref{eq-prob-1}, if the system is prepared in
 the state $\xi$, it is registered in a domain $U_j$ with probability
 \begin{equation}\label{eq-prob-2}
\int_M f_j(z) \; d\sigma_\xi \;.
\end{equation}
This statistical procedure can be described by an $\cL(H)$-valued POVM $A=\{A_1,...,A_N\}$
on $\Omega_N$, where each $A_j$ is the operator of multiplication
by $f_j$. Treating $A$ as an observable, we see that at a state $\xi \in H$ it accepts value $j \in \Omega_N$ with the probability $$\langle A_j \xi, \xi \rangle = \int_M f_j |\xi|^2 \; d\mu = \int_M f_j d\sigma_{\xi}\;,$$
which agrees with the rule \eqref{eq-prob-2} above.

Clearly, the POVM $A$ is commutative, but not sharp. For instance, if at least one of the sets, say, $U_1$
is not completely contained in the union of the others and does not coincide with $M$, there exist points
$z \in U_1$ and $w \in M \setminus U_1$ with $f_1(z) = 1$ and $f_1(w)=0$. Put $x_* = (1,0,...,0) \in \R^N$ and observe that the noise operator
$\Delta_A (x_*)$ acts on $H$ by multiplication on $f_1-f_1^2$. Since $f_1$ necessarily takes the value $1/2$ we conclude
that the noise $\cN(A)$ satisfies
$$\cN(A) \geq ||\Delta_A (x_*)||_{op} \geq 1/4\;.$$

Since $A$ is commutative, it is a smearing of a sharp observable. The latter can be chosen
as the following canonical $\cL(H)$-valued projector valued measure $P$ on the phase space $M$. Given
a Borel subset $X \subset M$, the projector $P(X)$ is the operator of multiplication by the indicator
function $\chi_X$ of $X$. The equation
$$A_j = \int_M f_j(z) \; dP(z)$$ shows that $A$
is the smearing of $P$ with the Markov kernel given
by the partition of unity $\{f_j\}$ (see \eqref{eq-smearing-int} above). Thus the systematic noise $\cN_s(A)$  vanishes.

Finally, let us describe the quantum version of our registration procedure: Let $(M^{2n},\omega)$ be a closed
symplectic manifold equipped with the phase volume $\omega^n$, and let
$T_m : C^\infty(M) \to \cL(H_m)$ be a Berezin-Toeplitz quantization. Consider the registration
procedure associated to an open cover $\{U_1,...,U_N\}$ of $M$ and a subordinated partition
of unity $\{f_1,...,f_N\}$.  The {\it quantum registration} is described by
the $\cL(H_m)$-valued POVM $A^{(m)}$ formed by the Toeplitz operators $T_m(f_j)$, $j=1,...,N$.
In a pure state $\xi \in H_m$, $|\xi|=1$ the quantum system is registered in a subset
$U_j$ with probability $\langle T_m(f_j)\xi,\xi \rangle$. With this language our main result
 given in Theorem \ref{thm-main} above states that if the cover is dominated by a displaceable subset $U \subset M$, the systematic noise of $A^{(m)}$ is bounded away from $0$ for all sufficiently large $m$.

 The registration procedure presented above can be considered as an attempt to localize the mechanical system in the phase space. We conclude that {\it any phase space localization
 of the quantized system  beyond certain scale yields a systematic noise.} Interestingly enough,
 this scale is governed by the notion of displaceability provided by symplectic topology.

\section{Discussion and further directions}\label{sec-disc}

\medskip
\noindent
{\sc From functions to subsets:} It is currently unclear whether a statement similar to Theorem \ref{thm-main} holds for {\it partitions} of a symplectic manifold into sufficiently small sets
(say, dominated by the same open displaceable subset). More precisely,
let $M = \bigsqcup_{j=1}^N X_j$ be such a partition,
and let $G_m$ be the POVM associated by \eqref{eq-POVM_Toeplitz} to the Berezin-Toeplitz quantization.
{\it Is it true that POVMs $\{G_m(X_j)\}_{j=1,...,N}$ on $\Omega_N$ are necessarily non-commutative for large $m$}? The answer is unknown to me even for the case when $M$ is the two-dimensional sphere. The difficulty here is
due to the fact that the properties of the Toeplitz operators $T_m(f)$ are much less understood for non-smooth
functions $f$, in particular for the indicator functions of subsets. We refer to \cite{Palamodov} for
a discussion (in a somewhat different setting) on spectral properties of Toeplitz operators associated to subsets of the standard symplectic plane.

\medskip
\noindent
{\sc Unsharpness principle:}
What is the precise value of the numerical
constant in the unsharpness principle formulated in Theorem \ref{cor-uncert-sm}?
It may well happen that this constant depends on the dimension of the Hilbert space $H$.

\medskip
\noindent {\sc Robustness of the systematic noise:} Given a POVM $A$, is its systematic noise $\cN_s(A)$ robust
with respect to small perturbations of $A$ in some natural metric? The difficulty here is that the systematic
noise is defined through all ``unsmearings" of $A$, while the process of smearing, being a diffusion, is in general not reversible. Thus there is no obvious relation between
 unsmearings of a POVM $A$ and of its small perturbations. On the contrary, the magnitude of non-commutativity
 $\nu_q(A)$ seems to be robust. If so, the unsharpness principle formulated in Theorem \ref{cor-uncert-sm}
 provides a robust lower bound for $\cN_s(A)$. It would be interesting to explore this in further details.

\medskip
\noindent
{\sc From quantum to classical and back:} Incidentally, the notion of a symplectic quasi-state, whose version is used in the proof of rigidity of partitions of unity in Section \ref{sec-rigpu} above, is closely related to the Gleason theorem \cite{Gleason}. The latter serves as an important argument against existence of non-contextual hidden variables in quantum mechanics. Roughly speaking, Gleason's theorem prohibits non-linear quantum quasi-states,
however their classical counterparts (defined via the correspondence principle) emerge in the large quantum number limit and can be detected by means of Floer theory. We refer to \cite{EPZ-physics} for a detailed exposition. From this perspective, the discussion of the present note starts and ends in quantum mechanics,
with a detour to the classical one.

\medskip
\noindent
{\sc Unsharpness of joint measurements:} In \cite{P-joint} we,
following a suggestion by Paul Busch, extend the results of the present paper to the case
of joint quantum measurements. We start with a pair of partitions of unity subordinated to finite open covers of $M$, and look at their images under the Berezin-Toeplitz quantization.
The corresponding POVMs are necessarily jointly measurable, that is arise as marginals of a joint observable. It turns out that in certain situations the relative geometry of the covers guarantees that for sufficiently large quantum numbers this joint observable is necessarily unsharp, that is possesses a systematic quantum noise.

\medskip
\noindent
{\sc Quantum dynamics:}
In the present note we quantized functions and the norm of their Poisson bracket.
In fact, one can quantize  Hamiltonian flows: Under the Berezin-Toeplitz quantization, the
Hamiltonian evolution on the symplectic manifold
corresponds to the Schr\"{o}dinger evolution on spaces $H_m$ up to an error of the order $O(1/m)$
as $m \to \infty$. One can give an elementary proof of this statement following the argument
of \cite[Proposition 2.7.1]{L} (cf. \cite[Remark (1), p.291]{BMS} and \cite[Corollary 8.3]{Charles}).
Another facet of the same phenomenon is that up to an error which tends to zero in the classical limit
``the coherent states move along the laws of classical mechanics", \cite[Remark (2), p.294]{BMS}. Therefore, it should be possible to translate into the quantum language
certain ``rigid" results on symplectic intersections, such as non-displaceability of a given closed subset.
This may enable one to detect meaningful footprints of symplectic rigidity in quantum dynamics.

\medskip

Finally, it would be interesting to explore other links between ``hard" symplectic topology and quantum mechanics. A step in this direction has been made is a series of papers by M.~de Gosson, see
\cite{deG} for a survey of intriguing interrelations between symplectic capacities and the uncertainty principle.

\bigskip

{\bf Acknowledgements:} I thank Tatyana Barron, Strom Borman, Paul Busch, Maurice de Gosson, Michael Entov, Ezra Getzler, Victor Guillemin, Yael Karshon, Yaron Ostrover, Iosif Polterovich and Frol Zapolsky for stimulating discussions and/or critical remarks on earlier drafts of the paper. Michael Entov and Victor Guillemin brought my attention to references \cite{MM} and \cite{Gu} respectively.
I am grateful to Paul Busch for illuminating comments on the noise operator. I thank the referees for numerous
valuable suggestions.

\begin{tabular}{l}
Leonid Polterovich\\
Department of Mathematics\\
University of Chicago\\
Chicago, IL 60637, USA\\
and\\
School of Mathematical Sciences\\
Tel Aviv University\\
Tel Aviv 69978, Israel\\
polterov@runbox.com\\
\end{tabular}

\end{document}